\documentclass[11pt,reqno]{article}
\usepackage{latexsym, amsmath, amssymb, amsthm, a4, epsfig}
\usepackage{graphicx}

\newtheorem{theorem}{Theorem}[section]

\newtheorem{lemma}[theorem]{Lemma}
\newtheorem{prop}[theorem]{Proposition}

\newcommand{\ds}{\displaystyle}
\newcommand{\p}{\partial}

\newcommand{\eqnref}[1]{(\ref {#1})}

\newcommand{\Cbb}{\mathbb{C}}

\newcommand{\Rbb}{\mathbb{R}}

\newcommand{\la}{\langle}
\newcommand{\ra}{\rangle}

\newcommand{\Hcal}{\mathcal{H}}
\newcommand{\Lcal}{\mathcal{L}}

\newcommand{\Scal}{\mathcal{S}}

\def\BA{{\bf A}}
\def\BB{{\bf B}}

\def\BI{{\bf I}}

\def\BK{{\bf K}}

\def\BR{{\bf R}}

\def\BT{{\bf T}}

\newcommand{\Ga}{\alpha}
\newcommand{\Gb}{\beta}
\newcommand{\Gd}{\delta}

\newcommand{\Gvf}{\varphi}

\newcommand{\Gl}{\lambda}

\newcommand{\Gm}{\mu}

\newcommand{\Gs}{\sigma}

\newcommand{\GG}{\Gamma}

\newcommand{\GO}{\Omega}

\newcommand{\BGG}{{\bf \GG}}

\newcommand{\beq}{\begin{equation}}
\newcommand{\eeq}{\end{equation}}

\def\ol{\overline}
\newcommand{\hatna}{\widehat{\nabla}}

\numberwithin{equation}{section}
\numberwithin{figure}{section}

\begin{document}

\title{Elastic Neumann--Poincar\'e operators on three dimensional smooth domains: Polynomial compactness and spectral structure\thanks{\footnotesize This work is supported by A3 Foresight Program among China (NSF), Japan (JSPS), and Korea (NRF 2014K2A2A6000567). Work of HK is supported by NRF 2016R1A2B4011304}}

\author{Kazunori Ando\thanks{Department of Electrical and Electronic Engineering and Computer Science, Ehime University, Ehime 790-8577, Japan. Email: {\tt ando@cs.ehime-u.ac.jp}.} \and Hyeonbae Kang\thanks{Department of Mathematics, Inha University, Incheon 402-751, S. Korea. Email: {\tt hbkang@inha.ac.kr}.} \and Yoshihisa Miyanishi\thanks{Center for Mathematical Modeling and Data Science, Osaka University, Osaka 560-8631, Japan. Email: {\tt miyanishi@sigmath.es.osaka-u.ac.jp}.}}

\date{}
\maketitle

\begin{abstract}
We prove that the elastic Neumann--Poincar\'e operator defined on the smooth boundary of a bounded domain in three dimensions, which is known to be non-compact, is in fact polynomially compact. As a consequence, we prove that the spectrum of the elastic Neumann-Poincar\'e operator consists of three non-empty sequences of eigenvalues accumulating to certain numbers determined by Lam\'e parameters. These results are proved using the surface Riesz transform, calculus of pseudo-differential operators and the spectral mapping theorem.
\end{abstract}




\section{Introduction}

The Neumann--Poincar\'e (abbreviated by NP) operator is a boundary integral operator which appears naturally when solving classical boundary value problems using layer potentials. Its study (for the Laplace operator) goes back to C. Neumann \cite{Neumann-87} and Poincar\'e \cite{Poincare-AM-87} as the name of the operator suggests. If the boundary of the domain, on which the NP operator is defined, is smooth, then the NP operator is compact. So the Fredholm integral equation, which appears when solving Dirichlet or Neumann problems, can be solved using the Fredholm index theory \cite{Fredholm-03}. If the domain has corners, the NP operator is not any more a compact operator, but a singular integral operator. The solvability of the corresponding integral equation was established in \cite{Verch-JFA-84}.

Recently interest in the NP operator, especially in the spectral properties, is growing rapidly, which is due to its connection to plasmon resonance and anomalous localized resonance on meta materials of negative dielectric constants. These resonances occur at eigenvalues and at the accumulation point of eigenvalues of the NP operator, respectively (see \cite{ACKLM-ARMA-13, MFZ-PR-05} and references therein). The spectral nature of the NP operator is also related to stress concentration between hard inclusions \cite{BT-ARMA-13}.

Regarding spectral properties of the NP operator, it is proved in \cite{KPS-ARMA-07} that the NP operator can be realized as a self-adjoint operator by introducing a new inner product on the $H^{-1/2}$-space (see also \cite{KKLSY-JLMS-16}), and so the NP spectrum consists of continuous spectrum and discrete spectrum (and possibly the limit points of discrete spectrum). If the domain has the smooth boundary, then the spectrum consists of eigenvalues converging to $0$. We refer to \cite{AKM-arXiv, MS-TAMS-17} for progress on the convergence rate of NP eigenvalues. If the domain has corners, the corresponding NP operator may exhibit a continuous spectrum (as well as eigenvalues). For recent development in this direction we refer to \cite{HKL-AIHP-17, KLY, PP-JAM-14, PP-arXiv}.

The NP operator for the Lam\'e operator (elastic NP operator) also appears naturally when solving the boundary value problems for the Lam\'e system \cite{DKV-Duke-88, Kup-book-65}. It is equally important to develop the elastic NP spectral theory for investigation of the elasticity analogy of plasmon resonance and stress concentration. In a recent paper \cite{AJKKY-arXiv} it is shown that the elastic NP operator can be symmetrized by introducing a new inner product on the $H^{-1/2}$-space using the single layer potential. However, there is a significant difference between the electro-static NP operator and the elastic one: The elastic one is \emph{not} compact unlike the electro-static one \cite{DKV-Duke-88}. So it is not clear how the elastic NP spectrum looks like. But, it turns out that the elastic NP operator in two dimensions is polynomially compact if the domain is smooth \cite{AJKKY-arXiv}. An operator $\BA$ is said to be polynomially compact if there is a polynomial $p$ such that $p(\BA)$ is compact. As a consequence it is proved in the same paper that the elastic NP spectrum consists of eigenvalues converging to certain numbers determined by Lam\'e parameters (see the next section).

The purpose of this paper is to extend the results in two dimensions to three dimensions. We prove that the elastic NP operator is polynomially compact even in three dimensions and, as a consequence, that the spectrum consists of eigenvalues converging to three different numbers determined by Lam\'e parameters. We emphasize that the method for two dimensions cannot be applied to three dimensions since it uses the notion of harmonic conjugates and the Hilbert transform. The proof of this paper relies on the surface Riesz transform and calculus of pseudo-differential operators (abbreviated by $\psi$DO). Recently, eigenvalues of the elastic NP operator on three dimensional balls are derived explicitly in \cite{DLL-preprint}, which is in accordance with the results of this paper. 

This paper is organized as follows. In the next section we state the main results of this paper in a precise manner. In section \ref{sec:Riesz}, we show that the elastic NP operator can be well approximated by surface Riesz transforms. Section \ref{sec:symbol} is to compute principal symbols of relevant $\psi$DOs. Proofs of the main theorems are provided in section \ref{sec:proof}. This paper ends with a brief conclusion.

\section{Statements of results}\label{sec:state}

Let $\GO$ be a bounded domain in $\Rbb^3$ whose boundary $\p\GO$ is $C^\infty$-smooth. Let $(\Gl, \Gm)$ be the Lam\'e constants for $\GO$ satisfying the strong convexity condition: $\Gm > 0$ and $3\Gl + 2\Gm > 0$. The isotropic elasticity tensor $\Cbb = ( C_{ijkl} )_{i, j, k, l = 1}^3$ and the corresponding Lam\'e system $\Lcal_{\Gl,\Gm}$ are defined by
\beq
C_{ijkl} := \Gl \, \Gd_{ij} \Gd_{kl} + \mu \, ( \Gd_{ik} \Gd_{jl} + \Gd_{il} \Gd_{jk} )
\eeq
and
\beq
\Lcal_{\Gl,\Gm} u := \nabla \cdot \Cbb \hatna u = \Gm \Delta u + (\Gl + \Gm) \nabla \nabla \cdot u,
\eeq
where $\hatna$ denotes the symmetric gradient, namely,
\beq
\hatna u := \frac{1}{2} \left( \nabla u + \nabla u^T \right) \quad (T \mbox{ for transpose}). \nonumber
\eeq
The corresponding conormal derivative on $\p \GO$ is defined to be
\beq\label{conormal}
\p_\nu u := (\Cbb \hatna u) n = \Gl (\nabla \cdot u) n + 2\Gm (\hatna u) n \quad \mbox{on } \p \GO,
\eeq
where $n$ is the outward unit normal to $\p \GO$.

Let $\BGG(x) = \left( \Gamma_{ij}(x) \right)_{i, j = 1}^d$ be the Kelvin matrix of the fundamental solution to the Lam\'{e} operator $\Lcal_{\Gl, \mu}$, namely,
\beq\label{Kelvin}
  \Gamma_{ij}(x) =
    - \ds \frac{\Ga_1}{4 \pi} \frac{\Gd_{ij}}{|x|} - \frac{\Ga_2}{4 \pi} \ds \frac{x_i x_j}{|x|^3}, \quad x \neq 0,
\eeq
where
\beq
  \Ga_1 = \frac{1}{2} \left( \frac{1}{\mu} + \frac{1}{2 \mu + \Gl} \right) \quad\mbox{and}\quad \Ga_2 = \frac{1}{2} \left( \frac{1}{\mu} - \frac{1}{2 \mu + \Gl} \right).
\eeq
The NP operator for the Lam\'e system is defined by
\beq\label{BK}
\BK [f] (x) := \mbox{p.v.} \int_{\p \GO} \p_{\nu_x} {\bf \GG} (x-y) f(y) d \Gs(y) \quad \mbox{a.e. } x \in \p \GO.
\eeq
Here, p.v. stands for the Cauchy principal value, and the conormal derivative $\p_{\nu_x}\BGG (x-y)$ of the Kelvin matrix with respect to $x$-variables is defined by
\beq\label{kerdef}
\p_{\nu_x}\BGG (x-y) b = \p_{\nu_x} (\BGG (x-y) b)
\eeq
for any constant vector $b$.

Even though the NP operator $\BK$ is not self-adjoint with respect to the usual $L^2$-inner product, it can be symmetrized on $\Hcal:=H^{-1/2}(\p\GO)^3$ (the Sobolev $-1/2$-space on $\p\GO$) by Plemelj's symmetrization principle (see \cite{AJKKY-arXiv}). So, the NP operator, as a self-adjoint operator on the Hilbert space $\Hcal$, may have continuous and discrete spectra.

In the electro-static case, the NP operator on smooth domains ($C^{1,\Ga}$-smooth to be precise) is compact, and its spectrum consists of eigenvalues (of finite multiplicities) converging to $0$. However, it is known that the elastic NP operator is \emph{not} compact even on smooth domains, which is due to the off-diagonal entries of the Kelvin matrix $\BGG(x)$ (see \cite{DKV-Duke-88}). So, it is not clear how the spectrum of the elastic NP operator looks like. As mentioned before, it is proved in \cite{AJKKY-arXiv} that the elastic NP operator in two dimensions is \emph{polynomially compact}, even though not compact. In fact, it is proved that if the polynomial $p(t)$ is given by $p_2(t)=t^2 - k_0^2$, where the constant $k_0$ is given by
\beq\label{kzero}
k_0= \frac{\mu}{2(2\mu+\Gl)},
\eeq
then $p_2(\BK)$ is a compact operator. As a consequence it is shown that the elastic NP spectrum in two dimensions consists of two non-empty sequences of eigenvalues converging to $k_0$ and $-k_0$, respectively.

The three dimensional case is not resolved yet since the method in two dimensions relies on the existence of conjugate harmonic functions and cannot be applied to three dimensions. In this paper we develop a new method to show polynomial compactness of $\BK$ which can be applied to three dimensions (as well as two dimensions). We obtain the following theorem.
\begin{theorem}\label{Polynomially compact}
Let $\GO$ be a smooth bounded domain in $\Rbb^3$, $\BK$ be the elastic NP operator on $\p\GO$, and let $p_3(t)=t(t^2-k_0^2)$ where $k_0$ is given by \eqnref{kzero}.
Then $p_3(\BK)$ is compact. Moreover, $\BK (\BK-k_0 \BI)$, $\BK (\BK + k_0 \BI)$ and $\BK^2 - k_0^2 \BI$ are \emph{not} compact.
\end{theorem}

As a consequence of Theorem \ref{Polynomially compact} and the spectral mapping theorem (see \cite{RS-book-80}), we obtain the following theorem on the spectral structure of the elastic NP operator.

\begin{theorem}\label{mainthm1}
Let $\GO$ be a smooth bounded domain in $\Rbb^3$ and $\BK$ be the elastic NP operator on $\p\GO$. The spectrum of $\BK$ consists of three non-empty sequences of eigenvalues  which converge to $0$, $k_0$ and $-k_0$, respectively.
\end{theorem}

Let us briefly describe the idea to prove Theorem \ref{Polynomially compact}. If $\GO$ is the half-space (even if it is not a bounded domain), then the non-compact part of $\BK$ consists of Riesz transforms (see \eqnref{halfspace}). Since the sum of squares of the Riesz transforms is $-\BI$, one can easily see that $\BK (\BK^2 - k_0^2 \BI)$ is compact. For bounded domains we show that the non-compact part of $\BK$ can be written in terms of surface Riesz transforms which is defined by the metric tensor of the surface $\p\GO$. We then show that the sum of squares of some variances of the surface Riesz transforms is $-\BI$ modulo a compact operator using calculus of pseudo-differential operators.

\section{NP operators and surface Riesz transforms}\label{sec:Riesz}

Straightforward computations using the definition \eqnref{kerdef} yield that
\beq
\p_{\nu_x}\BGG(x-y)= k_0 \BK_1(x,y) - \BK_2(x,y),
\eeq
where
\begin{align}
\BK_1(x,y) &= \frac{n_x (x-y)^T - (x-y) n_x^T}{2\pi |x-y|^{3}} ,  \\
\BK_2(x,y) &= \frac{\mu}{2\mu+\Gl} \frac{(x-y) \cdot n_x }{4\pi |x-y|^3} \BI + \frac{2(\mu+ \Gl)}{2\mu+\Gl} \frac{( x-y) \cdot n_x  }{4\pi |x-y|^{5}} (x-y)(x-y)^T .
\end{align}
Here $\BI$ is the $d \times d$ identity matrix. Since $\p\GO$ is smooth, we have
$$
|(x-y) \cdot n_x| \le C |x-y|^2
$$
for some constant $C$. So $\BK_2$ is weakly singular, namely,
$$
|\BK_2(x, y)| \le C |x-y|^{-1},
$$
and hence the integral operator defined by $\BK_2$ is compact on $\Hcal$. Let
\beq
\BT [f](x):= \text{p.v.} \int_{\p\GO} \BK_1(x,y) f(y) \, d \Gs(y), \quad x \in \p\GO.
\eeq
Then we have
\beq\label{BKBT}
\BK \equiv k_0 \BT.
\eeq
Here and throughout the expression $\BA \equiv \BB$ for operators $\BA$ and $\BB$ on $\Hcal$ indicates that $\BA-\BB$ is compact on $\Hcal$. We emphasize that $\BT$ is a singular integral operator and bounded on $\Hcal$ as well as on $L^2(\p\GO)^3$ (see \cite{CMM-AM-82}).

Denoting $n_x=(n_1(x), n_2(x), n_3(x))^T$, we have
\beq
\BK_1(x,y) = \frac{1}{2\pi |x-y|^{3}}
\begin{bmatrix}
0 & K_{12} (x, y) & K_{13} (x, y) \\
- K_{12} (x, y) & 0 & K_{23} (x, y) \\
- K_{13} (x, y) & - K_{23} (x, y) & 0
\end{bmatrix},
\eeq
where
\begin{align*}
K_{12} (x, y) &= n_1(x)(x_2-y_2) - n_2(x) (x_1-y_1), \\
K_{13} (x, y) &= n_1(x)(x_3-y_3) - n_3(x) (x_1-y_1), \\
K_{23} (x, y) &= n_2(x)(x_3-y_3) - n_3(x) (x_2-y_2) .
\end{align*}
Let
\beq
T_{ij}[f] (x) := \int_{\p\GO} \frac{K_{ij} (x, y)}{2\pi |x-y|^{3}} f(y) d\Gs(y),
\eeq
so that
\beq
\BT =
\begin{bmatrix}
0 & T_{12} & T_{13} \\
- T_{12} & 0 & T_{23} \\
- T_{13} & - T_{23} & 0
\end{bmatrix}.
\eeq

It is helpful to look into the case when $\GO$ is the upper half-space even if it is not a bounded domain. In this case, $n_x= (0,0,-1)^T$, and hence
\beq\label{halfspace}
\BT = \begin{bmatrix}
0 & 0 & R_1 \\
0 & 0 & R_2 \\
- R_1 & - R_2 & 0
\end{bmatrix},
\eeq
where $R_j$ is the Riesz transform, {\it i.e.},
$$
R_j [f](x_1,x_2) = \frac{1}{2\pi} \int_{\Rbb^2} \frac{x_j-y_j}{[(x_1-y_1)^2 + (x_2-y_2)^2]^{3/2}} f(y_1,y_2) \, dy_1dy_2, \quad j=1,2.
$$
Since $R_1^2+R_2^2=-I$ (see \cite{SW-book}), one can see easily from \eqnref{halfspace} that $\BT^3-\BT=0$.

We show that the operator $\BT$ on a bounded surface $\p\GO$ can be well approximated by the surface Riesz transforms which will be defined later in \eqnref{Rgdef}. To do so, let $U$ be a coordinate chart in $\p\GO$ so that there is an open set $D$ in $\Rbb^2$ and a parametrization $\Phi: D \to U$, namely,
\beq
x= \Phi(u)= (\Gvf_1(u), \Gvf_2(u), \Gvf_3(u)), \quad x \in U, \ \ u \in D.
\eeq
Then the metric tensor of the surface, denoted by $G(u)= (g_{ij}(u))_{i,j=1}^2$, is given by
\begin{align*}
dx_1^2+dx_2^2+dx_3^2
=g_{11} du_1^2 + 2 g_{12} du_1 du_2 + g_{22} du_2^2,
\end{align*}
where
\beq
g_{11}= |\p_1 \Phi|^2, \quad g_{12}=g_{21}=\p_1\Phi \cdot \p_2 \Phi, \quad g_{22}=|\p_2 \Phi|^2.
\eeq
Here and afterwards, $\p_j$ denotes the $j$-th partial derivative. In short, we have
\beq
G(u)= D\Phi(u)^T D\Phi(u),
\eeq
where $D\Phi$ is the ($3 \times 2$) Jacobian matrix of $\Phi$.

If $x, y\in U$ are given by $x=\Phi(u)$ and $y=\Phi(v)$, then one can see from the Taylor expansion that
\begin{align*}
|x-y|^2 &= |\Phi(u)-\Phi(v)|^2 \\
& = \la D\Phi(u)(u-v), D\Phi(u) (u-v) \ra +O(|u-v|^3).
\end{align*}
In short, we have
\beq
|x-y|^2 = \la u-v, G(u) (u-v) \ra +O(|u-v|^3).
\eeq
Thus we have
\beq\label{BxBy-3}
|x-y|^{-3} = L(u, u-v)+ O(|u-v|^{-2}),
\eeq
where
\beq
L(u, u-v) = \la u-v, G(u) (u-v) \ra^{-3/2}.
\eeq

We now show that if $x=\Phi(u)$, $y=\Phi(v) \in U$, then the following identities hold modulo $O(|u-v|^2)$ terms
\begin{align}
K_{12}(\Phi(u), \Phi(v)) &= \frac{(g_{11} \p_2 \Gvf_3  - g_{12} \p_1 \Gvf_3)(u) (u_1-v_1) - (g_{22} \p_1 \Gvf_3 - g_{21}\p_2\Gvf_3)(u)(u_2-v_2)}{|(\p_1 \Phi \times \p_2 \Phi)(u)|}, \label{K12} \\
K_{13}(\Phi(u), \Phi(v)) &= \frac{(g_{11} \p_2 \Gvf_2  - g_{12} \p_1 \Gvf_2)(u) (u_1-v_1) - (g_{22} \p_1 \Gvf_2 - g_{21}\p_2\Gvf_2)(u)(u_2-v_2)}{|(\p_1 \Phi \times \p_2 \Phi)(u)|}, \label{K13} \\
K_{23}(\Phi(u), \Phi(v)) &= \frac{(g_{11} \p_2 \Gvf_1  - g_{12} \p_1 \Gvf_1)(u) (u_1-v_1) - (g_{22} \p_1 \Gvf_1 - g_{21}\p_2\Gvf_1)(u)(u_2-v_2)}{|(\p_1 \Phi \times \p_2 \Phi)(u)|}.  \label{K23}
\end{align}

To prove above identities, we first note that the unit normal vector $n_x$ is given by $(\p_1 \Phi \times \p_2 \Phi)(u)/ |(\p_1 \Phi \times \p_2 \Phi)(u)|$.
So we have, modulo $O(|u-v|^2)$ terms,
\begin{align*}
& |(\p_1 \Phi \times \p_2 \Phi)(u)| K_{12}(\Phi(u), \Phi(v)) \\
&= (\Gvf_2(u)-\Gvf_2(v)) (\p_1 \Gvf_2 \p_2 \Gvf_3 - \p_1 \Gvf_3 \p_2 \Gvf_2)(u) \\
& \qquad - (\Gvf_1(u)-\Gvf_1(v))(\p_1 \Gvf_3 \p_2 \Gvf_1 - \p_1 \Gvf_1 \p_2 \Gvf_3)(u) \\
&= \nabla \Gvf_2(u) \cdot (u-v) (\p_1 \Gvf_2 \p_2 \Gvf_3 - \p_1 \Gvf_3 \p_2 \Gvf_2)(u) \\
& \qquad - \nabla \Gvf_1(u) \cdot (u-v)(\p_1 \Gvf_3 \p_2 \Gvf_1 - \p_1 \Gvf_1 \p_2 \Gvf_3)(u) \\
& = \left[ \p_2 \Gvf_3 ( (\p_1 \Gvf_1)^2 + (\p_1 \Gvf_2)^2)- \p_1 \Gvf_3 (\p_1 \Gvf_2 \p_2 \Gvf_2 + \p_1 \Gvf_1 \p_2 \Gvf_1) \right](u) (u_1 - v_1)  \\
& \qquad - \left[ \p_1 \Gvf_3 ( (\p_2 \Gvf_1)^2 + (\p_2 \Gvf_2)^2)- \p_2 \Gvf_3 (\p_1 \Gvf_2 \p_2 \Gvf_2 + \p_1 \Gvf_1 \p_2 \Gvf_1) \right](u) (u_2 - v_2)  \\
& = \left[ \p_2 \Gvf_3 ( g_{11} - (\p_1 \Gvf_3)^2 )- \p_1 \Gvf_3 (g_{12} - \p_1 \Gvf_3 \p_2 \Gvf_3 ) \right](u) (u_1 - v_1)  \\
& \qquad - \left[ \p_1 \Gvf_3 ( g_{22} - (\p_2 \Gvf_3)^2 )- \p_2 \Gvf_3 (g_{12} - \p_1 \Gvf_3 \p_2 \Gvf_3) \right](u) (u_2 - v_2)  \\
&= (g_{11} \p_2 \Gvf_3  - g_{12} \p_1 \Gvf_3)(u) (u_1-v_1) - (g_{22} \p_1 \Gvf_3 - g_{21}\p_2\Gvf_3)(u)(u_2-v_2).
\end{align*}
This proves \eqnref{K12}. The identities \eqnref{K13} and \eqnref{K23} can be proved similarly.

Choose open sets $U_j$ ($j=1,2$) in $\p\GO$ so that $\ol{U_1} \subset U_2$ and $\ol{U_2} \subset U$. Let $\chi_j$ ($j=1,2$) be a smooth functions such that $\chi_1=1$ in $U_1$, $\mbox{supp} (\chi_1) \subset U_2$, $\chi_2=1$ in $U_2$, and $\mbox{supp} (\chi_2) \subset U$. We denote by $M_j$ the multiplication operator by $\chi_j$, {\it i.e.},
\beq
M_j [f](x) = \chi_j(x) f(x).
\eeq
Then we have
\begin{align}
& (M_2 T_{ij} M_1)[f](\Phi(u)) \nonumber \\
&= \chi_2(\Phi(u)) \int_{\Rbb^2} \frac{K_{ij} (\Phi(u), \Phi(v))}{2\pi |\Phi(u)-\Phi(v)|^{3}} \chi_1(\Phi(v)) f(\Phi(v)) |(\p_1 \Phi \times \p_2 \Phi)(v)| \, dv. \label{M2TM1}
\end{align}
According to \eqnref{BxBy-3} and \eqnref{K12}, we have
\begin{align}
& \frac{K_{12} (\Phi(u), \Phi(v))}{|\Phi(u)-\Phi(v)|^{3}} |(\p_1 \Phi \times \p_2 \Phi)(v)| \nonumber \\
& = \left[ (g_{11} \p_2 \Gvf_3  - g_{12} \p_1 \Gvf_3)(u) (u_1-v_1) - (g_{22} \p_1 \Gvf_3 - g_{21}\p_2\Gvf_3)(u)(u_2-v_2) \right] L(u, u-v) \nonumber \\
& \qquad   + E(u, v), \label{K12-2}
\end{align}
where the error term $E(u, v)$ satisfies
\beq\label{Euv}
|E(u, v)| \lesssim |u-v|^{-1}.
\eeq

We now introduce the surface Riesz transform $R_j^g$, $j=1,2$: it is defined by
\beq\label{Rgdef}
R_j^g[f](u) =\frac{1}{2\pi}  \int_{\Rbb^2} {(u_j-v_j)}L(u, u-v) f(v) dv.
\eeq
It then follows from \eqnref{M2TM1} and \eqnref{K12-2} that
\begin{align*}
&(M_2 T_{12} M_1)[f](\Phi(u)) \\
&= \chi_2(\Phi(u)) \left[ (g_{11} \p_2 \Gvf_3  - g_{12} \p_1 \Gvf_3)(u) R_1^g - (g_{22} \p_1 \Gvf_3 - g_{21}\p_2\Gvf_3)(u) R_2^g \right] [ (\chi_1 \circ \Phi) (f \circ \Phi) ] \\
& \quad + \chi_2(\Phi(u)) \int_{\Rbb^2} E(u,v) \chi_1(\Phi(v)) f(\Phi(v))  \, dv.
\end{align*}
Let $\widetilde{M}_j$ be the multiplication operator by $\chi_j(\Phi(u))$ for $j=1,2$, and define $\tau$ by
\beq
\tau[f](u)= f(\Phi(u)).
\eeq
Since the operator defined by $E(u,v)$ is compact thanks to \eqnref{Euv}, we have
\beq\label{T12equiv}
\tau M_2 T_{12} M_1 \equiv \widetilde{T}_{12} \tau,
\eeq
where
\beq
\widetilde{T}_{12} := \widetilde{M_2} \left[ (g_{11} \p_2 \Gvf_3  - g_{12} \p_1 \Gvf_3) R_1^g - (g_{22} \p_1 \Gvf_3 - g_{21}\p_2\Gvf_3) R_2^g \right] \widetilde{M_1} .
\eeq
Likewise, we obtain
\beq\label{T13equiv}
\tau M_2 T_{13} M_1 \equiv \widetilde{T}_{13} \tau \quad\text{and}\quad
\tau M_2 T_{23} M_1 \equiv \widetilde{T}_{23} \tau,
\eeq
where
\begin{align}
\widetilde{T}_{13} &:= \widetilde{M_2} \left[ (g_{11} \p_2 \Gvf_2  - g_{12} \p_1 \Gvf_2) R_1^g - (g_{22} \p_1 \Gvf_2 - g_{21}\p_2\Gvf_2) R_2^g \right] \widetilde{M_1} , \\
\widetilde{T}_{23} &:= \widetilde{M_2} \left[ (g_{11} \p_2 \Gvf_1  - g_{12} \p_1 \Gvf_1) R_1^g - (g_{22} \p_1 \Gvf_1 - g_{21}\p_2\Gvf_1) R_2^g \right] \widetilde{M_1} .
\end{align}

\section{Principal symbols of $\widetilde{T}_{ij}$}\label{sec:symbol}

In this section we realize $\widetilde{T}_{ij}$ as a $\psi$DO and calculate its symbol.
The classical $\psi$DO is defined as
$$
Op(\Gs)[f](u)=\frac{1}{(2\pi)^2}  \int_{\Rbb^2} \int_{\Rbb^2} \Gs(u, \xi) e^{i(u-v)\cdot \xi} f(v) \, dv d\xi,
$$
where $\Gs(u, \xi)$ is called the symbol of $Op(\Gs)$.

We first compute the symbol of the surface Riesz transform:
\begin{lemma}\label{symbolRG}
Let $G(u)^{-1}=(g^{ij}(u))$ be the inverse of the metric tensor $G=(g_{ij})$. Then the symbol $p_j(u, \xi)$ \emph{($u \in D$)} of $R_j^g$ is given by
\beq
p_j(u, \xi)= \frac{-i}{\sqrt{\det (g_{jk}(u))}}\frac{\sum_{k} g^{jk}(u)\xi_k}{\sqrt{\sum_{i, j} g^{jk}(u)\xi_j \xi_k}} , \quad j=1,2.
\eeq
\end{lemma}

\begin{proof}
We obtain from the definition of the surface Riesz transform \eqnref{Rgdef} that
\begin{align*}
\frac{1}{2\pi} \int_{\Rbb^2} (u_j-v_j)L(u, u-v) f(v) dv &=\frac{1}{(2\pi)^2}  \int_{\Rbb^2} \int_{\Rbb^2} e^{i(u-v)\cdot \xi} p_j(u, \xi) d\xi f(v) dv  \\
&= \frac{1}{(2\pi)^2} \int_{\Rbb^2} e^{iu\cdot \xi} p_j(u, \xi) \hat{f}(\xi) d\xi.
\end{align*}
So $p_j(u, \xi)$ is the Fourier transform of $\frac{1}{2\pi} z_j L(u, z)$ with respect to $z$-variables, namely,
\begin{align}
p_j(u, \xi) &= \frac{1}{2\pi} \int_{\Rbb^2} z_j L(u, z) e^{-iz \cdot \xi} dz \nonumber \\
& = \frac{1}{2\pi} \int_{\Rbb^2} \frac{z_j}{\la z, G(u)z \ra^{3/2}} e^{-iz \cdot \xi} dz. \label{symbol}
\end{align}

The metric tensor is diagonalizable by an orthogonal matrix, that is, there is an orthogonal matrix $P(u)$ and functions $\Ga(u)$ and $\Gb(u)$ such that
\beq
P^{-1}(u) G(u) P(u)=
\begin{bmatrix}
\Ga^2(u) &  0 \\
 0 & \Gb^2(u)
\end{bmatrix}
\eeq
or
\beq
\begin{bmatrix}
\Ga^{-1}(u) &  0 \\
 0 & \Gb^{-1}(u)
\end{bmatrix}
P^{-1}(u) G(u) P(u)
\begin{bmatrix}
\Ga^{-1}(u) &  0 \\
 0 & \Gb^{-1}(u)
\end{bmatrix}=
\begin{bmatrix}
1 &  0 \\
0 & 1
\end{bmatrix}.
\eeq
Let
$$
Q(u): =P(u)
\begin{bmatrix}
\Ga^{-1}(u) &  0 \\
 0 & \Gb^{-1}(u)
\end{bmatrix},
$$
and introduce new variables $w$ by
\beq
z=Q(u)w.
\eeq
Note that $\Ga(u)\Gb(u)=\sqrt{\det (g_{jk}(u))}$ and $\la z, G(u)z \ra^{3/2}=|w|^3$. Thus we have
\eqnref{symbol} that
\begin{align*}
p_j(u, \xi)&=\int_{\Rbb^2} \frac{1}{2\pi} \frac{[Q(u) w]_j}{|w|^3}
e^{-i[Q(u) w] \cdot \xi} \frac{dw}{\Ga(u)\Gb(u)} \\
&=\frac{1}{\sqrt{\det (g_{jk}(u))}} \int_{\Rbb^2}
\frac{1}{2\pi} \frac{[Q(u) w]_j}{|w|^3}
e^{-iw \cdot [Q(u)^{T}  \xi]} {dw} \\
&=\frac{-i}{\sqrt{\det (g_{jk}(u))}}  \frac{[Q(u) Q(u)^{T}  \xi]_j}{|Q(u)^{T}  \xi|} \\
&=\frac{-i}{\sqrt{\det (g_{jk}(u))}}\frac{\sum_{k} g^{jk}(u)\xi_k}{\sqrt{\sum_{j, k} g^{jk}(u)\xi_j \xi_k}}
\end{align*}
as desired.
\end{proof}

Lemma \ref{symbolRG} shows that the symbol of the operator $\widetilde{M}_2 R_j^g$ is given by $\chi_2(\Phi(u))p_j(u, \xi)$, which is of homogeneous of order $0$ in $\xi$-variable. If we mollify the homogeneous symbol, {\it i.e.}, if we let
$$
\widetilde{p}_j(u, \xi):= \chi_2(\Phi(u))p_j(u, \xi) \eta (\xi),
$$
where $\eta$ is a smooth function such that $\eta(\xi)=0$ for $|\xi| \le 1/2$ and $\eta(\xi)=1$ for $|\xi|>1$, then $\widetilde{p}_j$
belongs to $\Scal^{0}(\Rbb^2_x \times \Rbb^2_{\xi})$, the classical symbol class of order $0$. Furthermore, $Op(\widetilde{p}_j) - \widetilde{M}_2 R_j^g$ is a compact operator as one can see easily. In other words, $Op(\widetilde{p}_j) \equiv \widetilde{M}_2 R_j^g$.

So we may apply the product formula for calculus of $\psi$DO. In particular, if $\Gs_1$, $\Gs_2 \in \Scal^{0}(\Rbb^2_x \times \Rbb^2_{\xi})$, then the symbol of $Op(\Gs_1)Op(\Gs_2)$ is given by
\beq
\Gs_1(u,\xi)\Gs_2(u,\xi) + \Gs_3(u,\xi),
\eeq
where $\Gs_3 \in \Scal^{-1}(\Rbb^2_x \times \Rbb^2_{\xi})$ (see \cite{Shubin}). In particular, $Op(\Gs_3)$ is compact on $H^s$-space. So in notation of this paper we have
\beq
Op(\Gs_1)Op(\Gs_2)\equiv Op(\Gs_1 \Gs_2).
\eeq

\begin{lemma}\label{symbolT}
The principal symbols of $\widetilde{T}_{12}$, $\widetilde{T}_{13}$ and $\widetilde{T}_{23}$ are respectively given by
\begin{align}
\Gs_{12}(u, \xi):= \frac{-i\chi_1(u)\sqrt{\det(g^{jk}(u))}(\p_2 \Gvf_3(u) \xi_1 - \p_1\Gvf_3(u) \xi_2)}{\sqrt{\sum_{j, k} g^{jk}(u)\xi_j \xi_k}}, \\
\Gs_{13}(u, \xi):= \frac{-i\chi_1(u)\sqrt{\det(g^{jk}(u))}(\p_2 \Gvf_2(u) \xi_1 - \p_1\Gvf_2(u) \xi_2)}{\sqrt{\sum_{j, k} g^{jk}(u)\xi_j \xi_k}}, \\
\Gs_{23}(u, \xi):= \frac{-i\chi_1(u)\sqrt{\det(g^{jk}(u))}(\p_2 \Gvf_1(u) \xi_1 - \p_1\Gvf_1(u) \xi_2)}{\sqrt{\sum_{j, k} g^{jk}(u)\xi_j \xi_k}}.
\end{align}
\end{lemma}

\begin{proof}
We only give a proof for $\widetilde{T}_{12}$. The other two cases can be proved in the same way.

By the product rule of $\psi$DO and Lemma \ref{symbolRG}, the principal symbol of $\widetilde{T}_{12}$, after mollification, is given by
\begin{align*}
&\frac{-i\chi_2(\Phi(u))\chi_1(\Phi(u))}{\sqrt{\det (g_{jk}(u))}\sqrt{\sum_{j, k} g^{jk}(x)\xi_j \xi_k}} \Big[ (\p_2\Gvf_3(u) g_{11} - \p_1\Gvf_3(u) g_{12}) (g^{11}\xi_1+g^{12}\xi_2) \\
& \qquad\qquad\qquad + (\p_2\Gvf_3(u) g_{12} -\p_1\Gvf_3(u) g_{22})  (g^{21}\xi_1+g^{22}\xi_2) \Big] .
\end{align*}
Since $\chi_2\chi_1=\chi_1$ and
\begin{align*}
g^{11}=\frac{g_{22}}{\det(g_{ij}(x))}, \quad g^{12}=g^{21}=-\frac{g_{12}}{\det(g_{ij}(x))}, \quad
g^{22}=\frac{g_{11}}{\det(g_{ij}(x))},
\end{align*}
we see that the above formula equals to
\begin{align*}
&\frac{-i\chi_1(\Phi(u))\sqrt{\det (g_{jk}(u))}}{\sqrt{\sum_{j, k} g^{jk}(x)\xi_j \xi_k}} \Big[ (\p_2\Gvf_3(u) g^{22} + \p_1\Gvf_3(u) g^{12}) (g^{11}\xi_1+g^{12}\xi_2) \\
& \qquad\qquad\qquad + (-\p_2\Gvf_3(u) g^{12} -\p_1\Gvf_3(u) g^{11})  (g^{21}\xi_1+g^{22}\xi_2) \Big] \\
&= \frac{-i\chi_1(\Phi(u))\sqrt{\det (g_{jk}(u))}}{\sqrt{\sum_{j, k} g^{jk}(x)\xi_j \xi_k}} (g^{11} g^{22} -g^{12} g^{21}) (\p_2\Gvf_3(u)  \xi_1 - \p_1\Gvf_3(u) \xi_2) \\
&= \frac{-i\chi_1(\Phi(u)) \sqrt{\det(g^{jk}(u))}(\p_2\Gvf_3(u)  \xi_1 - \p_1\Gvf_3(u) \xi_2)}{\sqrt{\sum_{j, k} g^{jk}(u)\xi_j \xi_k}}.
\end{align*}
This completes the proof.
\end{proof}

Let $\Gs_{ij}$ be the symbols appeared in Lemma \ref{symbolT}, and let
\beq
X_{ij}:= Op(\Gs_{ij}).
\eeq
Lemma \ref{symbolT} shows that
\beq
\widetilde{T}_{ij} \equiv X_{ij}.
\eeq
Let
\beq\label{Rdef}
\BR =
\begin{bmatrix}
0 & X_{12} & X_{13} \\
-X_{12} & 0 & X_{12} \\
-X_{13} & -X_{12} & 0
\end{bmatrix}.
\eeq
We then infer from \eqnref{T12equiv} and \eqnref{T13equiv} that
\beq\label{tauform}
\tau M_2 \BT M_1 \equiv \BR \tau.
\eeq

\section{Proofs of main theorems}\label{sec:proof}

We now prove Theorem \ref{Polynomially compact} and \ref{mainthm1}. We first prove the following lemma.

\begin{lemma}\label{BTM compact}
The operator $\BR^3 - \widetilde{M}_1^2 \BR$ is compact on $H^{-1/2}(\Rbb^2)$, but $\BR(\BR - \widetilde{M}_1 \BI)$, $\BR(\BR + \widetilde{M}_1 \BI)$, and $\BR^2 - \widetilde{M}_1^2 \BI$ are not compact.
\end{lemma}
\begin{proof}
We see from the product formulas of $\psi$DO and the Cayley-Hamilton theorem that $\BR$ given by \eqnref{Rdef} satisfies
$$
\BR^3+(X_{12}^2+X_{13}^2+X_{13}^2) \BR \equiv 0.
$$
We also see from the product formulas of $\psi$DO that the principal symbol of $X_{12}^2+X_{13}^2+X_{13}^2$ is given by
\begin{align}
&\frac{-\chi_1(\Phi(u))^2 \det(g^{jk}(u))}{{\sum_{j, k} g^{jk}(u)\xi_j \xi_k}}
\sum_{j=1}^3 (\p_2 \Gvf_j(u) \xi_1 - \p_1\Gvf_j(u) \xi_2)^2 \nonumber \\
&=\frac{-\chi_1(\Phi(u))^2 \det(g^{jk}(u))}{{\sum_{j, k} g^{jk}(u)\xi_j \xi_k}}(g_{22} \xi_1^2- 2g_{12} \xi_1 \xi_2 + g_{11}\xi_2^2) \nonumber \\
&=\frac{-\chi_1(\Phi(u))^2}{{\sum_{j, k} g^{jk}(u)\xi_j \xi_k}}(g^{11} \xi_1^2+ 2 g^{12} \xi_1 \xi_2+ g^{22}\xi_2^2) \nonumber \\
&=-\chi_1(\Phi(u))^2 . \label{sumsquare}
\end{align}
So, $\BR^3 - \widetilde{M}_1^2 \BR$ is compact.

Note that
\beq
\BR^2 =
\begin{bmatrix}
-X_{12}^2-X_{13}^2 & -X_{13}X_{23} & X_{12}X_{23} \\
-X_{23}X_{13} &  -X_{12}^2-X_{23}^2 & -X_{13}X_{12} \\
X_{23}X_{12} &  -X_{13}X_{12} & -X_{23}^2-X_{13}^2
\end{bmatrix}.
\eeq
So, in view of \eqnref{sumsquare}, we see that the principal symbols of $\text{tr}\,\BR(\BR - \widetilde{M}_1 \BI)$, $\text{tr}\,\BR(\BR + \widetilde{M}_1 \BI)$, and $\text{tr}\,(\BR^2 - \widetilde{M}_1^2 \BI)$ are respectively given by $2\chi_1(\Phi(u))^2$, $2\chi_1(\Phi(u))^2$, and $\chi_1(\Phi(u))^2$. Here $\text{tr}$ stands for the trace. So, $\text{tr}\,\BR(\BR - \widetilde{M}_1 \BI)$, $\text{tr}\,\BR(\BR + \widetilde{M}_1 \BI)$, and $\text{tr}\,(\BR^2 - \widetilde{M}_1^2 \BI)$ are non-compact, and so are $\BR(\BR - \widetilde{M}_1 \BI)$, $\BR(\BR + \widetilde{M}_1 \BI)$, and $\BR^2 - \widetilde{M}_1^2 \BI$. This completes the proof.
\end{proof}

\begin{prop}\label{prop}
The operator $\BT^3-\BT$ is compact on $\Hcal$, but $\BT (\BT - \BI)$, $\BT (\BT + \BI)$ and $(\BT^2 - \BI)$ are not compact.
\end{prop}
\begin{proof}
According to \eqnref{tauform} and Lemma \ref{BTM compact}, we have
$$
\tau ((M_2 \BT M_1)^3 - M_1^2 (M_2 \BT M_1)) \equiv (\BR^3 - M_1^2 \BR) \tau \equiv 0.
$$
Thus we have
$$
(M_2 \BT M_1)^3 - M_1^2 (M_2 \BT M_1) \equiv 0.
$$
Since the commutator $[\BT, M_j]$ is compact and $M_2 M_1=M_1$, we have
$$
(M_2 \BT M_1)^3 - M_1^2 (M_2 \BT M_1) \equiv (\BT^3 - \BT) M_1^3,
$$
and hence $(\BT^3 - \BT) M_1^3$ is compact.

Recall that $\chi_1 =1$ in $U_1$. If $\{ f_n\}$ is a bounded sequence in $\Hcal$ whose supports lie in $U_1$, then
$$
(\BT^3 - \BT) [f_n] = (\BT^3 - \BT) M_1^3 [f_n]
$$
has a subsequence converging in $\Hcal$. Since $U_1$ is an arbitrary open set contained in a single coordinate chart, we may use the argument of partition of unity to show that $\BT^3-\BT$ is compact. In fact, let $\{ (U_k, \Gvf_k) \}_{k=1}^N$ be a partition of unity for $\p\GO$ such that each $U_k$ lies in a coordinate chart, and let $\{ f_n\}$ be a bounded sequence in $\Hcal$. Then
$$
(\BT^3 - \BT) [f_n] = \sum_{k=1}^N (\BT^3 - \BT) [\Gvf_k f_n].
$$
For each fixed $k$, $\mbox{supp}( \Gvf_k f_n) \subset U_k$. So, $(\BT^3 - \BT) [\Gvf_k f_n]$ has a convergent subsequence. Thus we can infer that there is a subsequence, say $\{ f_{n_l} \}$ such that $(\BT^3 - \BT) [\Gvf_k f_{n_l}]$ is convergent in $l$ for each $k$. So, $(\BT^3 - \BT) [f_{n_l}] = \sum_{k=1}^N (\BT^3 - \BT) [\Gvf_k f_{n_l}]$
is convergent. Hence, $\BT^3 - \BT$ is compact.

To show that $\BT (\BT - \BI)$ is not compact, we observe as before that
$$
\BT (\BT - \BI) M_1^2 \equiv (M_2 \BT M_1)^2 - M_1 (M_2 \BT M_1),
$$
and hence
$$
\tau \BT (\BT - \BI) M_1^2 \equiv (\BR^2 - \widetilde{M}_1 \BR)\tau.
$$
Since $(\BR^2 - \widetilde{M}_1 \BR)$ is non-compact by Lemma \ref{BTM compact}, so is $\BT (\BT - \BI) M_1^2$. So, $\BT (\BT - \BI)$ is non-compact.

Non-compactness of $\BT (\BT - \BI)$ and $\BT^2-\BI$ can be proved similarly.
\end{proof}

\noindent{\sl Proof of Theorem \ref{Polynomially compact}}.
Since $\BK \equiv k_0 \BT$  by \eqnref{BKBT} and $\BT^3-\BT \equiv 0$, we have
$$
\BK^3-k_0^2 \BK \equiv k_0^3 (\BT^3-\BT) \equiv 0.
$$
Moreover, since $\BT (\BT - \BI)$, $\BT (\BT + \BI)$ and $(\BT^2 - \BI)$ are non-compact,  so are $\BK (\BK-k_0 \BI)$, $\BK (\BK + k_0 \BI)$ and $\BK^2 - k_0^2 \BI$.
\qed

\medskip
\noindent{\sl Proof of Theorem \ref{mainthm1}}.
We first emphasize that $\BK$ is a self-adjoint operator on $\Hcal$ (see \cite{AJKKY-arXiv}). Denote the spectrum of $\BK$ by $\Gs(\BK)$. By the spectral mapping theorem, we have $p_3(\Gs(\BK))=\Gs(p_3(\BK))$. Since $p_3(K)$ is compact, $\Gs(p_3(\BK))$ consists of eigenvalues (of finite multiplicities) converging to $0$. So, $\Gs(\BK)$ is discrete eigenvalues and possible accumulation points $0$, $k_0$, and $-k_0$, which are zeros of $p_3(t)$.

We now show that there are actually non-empty sequences of eigenvalues converging to $0$, $k_0$, and $-k_0$, respectively. Suppose that there is no sequence of eigenvalues converging to $0$. Then all eigenvalues converge to either $k_0$ or $-k_0$. It implies that $\BK^2 - k_0^2 \BI$ is compact, which contradicts Proposition \ref{prop}. Similarly one can show that there are non-empty sequences of eigenvalues converging to $k_0$ and $-k_0$, respectively. This completes the proof. \qed

\section*{Conclusion}
We prove that the elastic NP operator on three dimensional bounded domains with smooth boundaries is polynomially compact, and the elastic NP spectrum consists of three non-empty sequences of eigenvalues accumulating to $0$ and $\pm k_0$.

We mention that results of this paper are obtained under the assumption that the boundary of the domain is $C^\infty$-smooth, while those in two dimension were proved for $C^{1,\Ga}$ domains. The smoothness assumption is required since the method of proofs of this paper uses calculus of $\psi$DO.
It is likely that the main results of this paper are valid for domains with $C^{1,\Ga}$ boundaries like the two dimensional case. To prove it, it is necessary to compute the compositions of surface Riesz potentials, which are singular integral operators. We will pursue this in future.


\end{document}